\documentclass[a4paper,12pt]{article}
\usepackage{enumerate,amsmath,amsfonts,bm,float}
\usepackage{latexsym,amssymb,amsthm, amsfonts, euscript}
\usepackage{amscd,enumerate,verbatim,calc}
\usepackage{longtable, scalefnt, mathrsfs}
\newtheorem{Exaexa}{Example}[section]

\numberwithin{equation}{section}
\newtheorem{thm}{Theorem}[section]
\theoremstyle{definition}

\newtheorem{rem}{Remark}[section]
\theoremstyle{plain}
\newtheorem{lem}[thm]{Lemma}

\begin{document}

\begin{center}{\large{\bf{Explicit factorization of $x^{2^nd}-1$ over a finite field}}}\end{center}
\begin{center} Manjit Singh\\\texttt{ manjitsingh.math@gmail.com}\\
Department of Mathematics\\
D.C.R. University of Science \& Technology,
Murthal-131039, India\end{center}

\begin{abstract} Let $\mathbb{F}_q$ be a finite field  of odd characteristic containing  $q$ elements and integer $n\ge 1$.  In this paper,   the explicit factorization of $x^{2^nd}-1$ over $\mathbb{F}_q$ is obtained when $d$ is an odd divisor of $q+1$.   
 \end{abstract}
\section{Introduction}
 Throughout this paper, $\mathbb{F}_q$ denotes a finite field with $q$ elements, where $q$ is  an odd prime power. Let $n$ be a positive integer such that $\gcd(n,q)=1$ and $\zeta$ be a primitive $n$th root of unity over $\mathbb{F}_q$. Then the polynomial $$\Phi_{n}(x)=\prod_{\substack{j=1\\\gcd(j,n)=1}}^{n}(x-\zeta^j)$$      is  called the $n$th cyclotomic polynomial  over $\mathbb{F}_q$.  The degree of $\Phi_n(x)$ is $\phi(n)$, where $\phi(n)$ is the Euler Totient function. Note that if $k$ divides $n$, then  each primitive $k$th root of unity is also an $n$th root  of unity. An important relation  between $k$th cyclotomic polynomial and $n$th root of unity is \begin{center}$\displaystyle{x^n-1=\prod_{k|n}\Phi_k(x)}$.\end{center}  Applying the M\"{o}bius inversion formula to the above equation, we find that \begin{center}$\displaystyle{\Phi_n(x)=\prod_{k|n}(x^k-1)^{\mu(\frac{n}{k})}=\prod_{k|n}(x^{\frac{n}{k}}-1)^{\mu(k)}}$,\end{center} where $\mu$ is the M\"{o}bius function. Moreover, if  $e$ is the least positive integer satisfying $q^e\equiv1\pmod n$, then $\Phi_n(x)$ splits into the product of $\phi(n)/e$ monic irreducible polynomials over $\mathbb{F}_q$ of degree $e$. In particular, $\Phi_n(x)$ is irreducible over $\mathbb{F}_q$ if and only if $e=\phi(n)$ (\cite{lidl,roman}). 

 Factoring polynomials over finite fields  is a classical topic of mathematics and play important roles in cryptography, coding theory and communication theory (see \cite{ber, gol, lidl}).   In recent years, the factorization of polynomials $x^n-1$ and $\Phi_n(x)$ have received extensive attention from researchers.  A brief description of some of the past accomplishments regarding the factorization of $x^{n}-1$ and $\Phi_{n}(x)$ over $\mathbb{F}_q$ is given below.

   The factorization of $\Phi_{2^n}(x)$ over $\mathbb{F}_q$ when $q\equiv1\pmod 4$ can be found in \cite{lidl, wan}. In  1996, Meyn  \cite{meyn} obtained the factorization of $\Phi_{2^n}(x)$ over $\mathbb{F}_q$ when $q\equiv3\pmod 4$. Fitzgerald and Yucas \cite{fitz7} studied the explicit factorization of $\Phi_{2^nr}(x)$ over $\mathbb{F}_q$, where $r$ is an odd prime such that $q\equiv\pm1\pmod r$. They also obtained explicit irreducible factors of $\Phi_{2^n3}(x)$. Wang and Wang \cite{ww} gave the explicit factorization of $\Phi_{2^n5}(x)$.     Stein \cite{stein} obtained the factors of $\Phi_{r}(x)$ when $q$ and  $r$ are distinct odd primes.  Further, assuming  the explicit factors of $\Phi_r(x)$ are known for  arbitrary odd integer $r>1$, Tuxanidy and Wang  \cite{tw} obtained the irreducible factors of $\Phi_{2^nr}(x)$ over $\mathbb{F}_q$. Singh and Batra  \cite{singh} revisited the factorization of $\Phi_{2^n}(x)$ and  introduced a recurrence relation to obtain  the coefficients of all irreducible factors  of $\Phi_{2^n}(x)$ over $\mathbb{F}_q$ when $q\equiv3\pmod 4$.

In 1993, Blake et al. \cite{bla} explicitly determined all the irreducible factors of  $x^{2^n}\pm1$ over $\mathbb{F}_p$, where $p$ is a prime with $p\equiv3\pmod 4$. Chen et al.  \cite{chen} obtained the explicit factorization of $x^{2^n p^l}-1$ over $\mathbb{F}_q$, where integer $l\ge 1$ and $p$ is an odd prime with $p|(q-1)$. In 2015, Mart\'{i}nez et al. \cite{bro}  generalized the results discussed in \cite{chen} by considering the explicit factorization of $x^{2^nd}-1$  over $\mathbb{F}_q$ when $d=p_1^{l_1}p_2^{l_2}\cdots p_r^{l_r}$, where integer $l_i\ge1$ and  $p_i$ is an  odd prime such that $p_i|(q-1)$ for each $i=1,2,\cdots, r$. It is natural to ask what happens for the case of factorizing $x^{2^nd}-1$ over a finite field $\mathbb{F}_q$ when $d$ divides $q+1$.

 In this paper, the factorization of $x^{2^nd}-1$ into the product of irreducible factors over $\mathbb{F}_q$ is obtained when $d$ is an odd divisor of $q+1$. This factorization  gives exact   information regarding the type of irreducible factors, the number of irreducible factors and  their degrees.   

 The  paper is organized as follows:  In Sect. 2, we introduce some basic results concerning cyclotomic polynomials and irreducibility of  decomposable polynomials over finite fields. This section also contains two recursive sequences  which are useful to obtain the coefficients of irreducible factors of $x^{2^nd}-1$ over $\mathbb{F}_q$. In Sect. 3, we find  irreducible factors of  the decomposable cyclotomic polynomials of the type $\Phi_{2^n}(x^d)$ over $\mathbb{F}_q$ from irreducible factors of some decomposable cyclotomic  polynomials of smaller orders.  In Sect. 4, we give the explicit factorization of $x^{2^nd}-1$ over $\mathbb{F}_q$ via factoring the cyclotomic polynomial $\Phi_{2^n}(x^d)$ into the product of irreducible factors over $\mathbb{F}_q$ in two different cases $q\equiv1\pmod{4}$ and $q\equiv3\pmod 4$.  In addition, two numerical examples are given to validate the results.  
\section{Preliminaries and auxiliary results}
The paper follows the standard notation of finite fields and the multiplicative group of a finite field $\mathbb{F}_q$, denoted  by $\mathbb{F}_q^*$. The following classical result presents necessary and sufficient conditions to verify the irreducibility of decomposable polynomials of the form $f(x^k)$ over finite fields.

\begin{lem}\cite[Theorem 3.35]{lidl} \label{tirr}Let $f_1(x), f_2(x),\cdots, f_N(x)$ be all distinct monic irreducible polynomials over $\mathbb{F}_q$ of degree $l$ and order $e$, and let $t\ge2$ be an integer whose prime factors divide $e$ but not $(q^l-1)/e$. Also assume that $4|(q^l-1)$ if $4|t$. Then $f_1(x^t),f_2(x^t),\cdots, f_N(x^t)$ are all distinct monic irreducible polynomials over $\mathbb{F}_q$ of degree $lt$ and order $et$. \end{lem}
 
\begin{lem}\cite[Exercise 2.57]{lidl}\label{cyclotomic} Suppose that  $p$ is an odd prime such that $\gcd(2p,q)=1$. Then, in $\mathbb{F}_q[x]$, the following properties of cyclotomic polynomials hold:\begin{enumerate}
\item[(i)] $\Phi_{2n}=\Phi_n(-x)$ if $n\ge3$ and $n$ is odd,
\item[(ii)] $\Phi_{np^k}(x)=\Phi_{np}(x^{p^{k-1}})$ for any positive integers $n,k$
\item[(iii)] $\displaystyle{\Phi_{np}(x)=\Phi_{n}(x^{p})}$ if $p$ divides $n$. In particular,  $\displaystyle{\Phi_{2^{k+r}}(x)=\Phi_{2^k}(x^{2^r})}$ for integers $k\ge1$ and $r\ge0$.
\item[(iv)] $\displaystyle{\Phi_{np}(x)=\Phi_{n}(x^{p})/\Phi_n(x)}$ if $p\nmid n$.
 \end{enumerate} \end{lem}

For a fixed integer $l\ge1$,  a reduced residue system modulo $l$  can be found by deleting   from $1, 2, \cdots, l$, a complete residue system modulo $l$, those integers that are not co-prime to $l$. Note that all reduced residue system modulo $l$ contains $\phi(l)$ integers. 
 \begin{lem}\label{res}\cite[Lemma 7.6]{koshy} Let $l$ be a positive integer and $a$ be any integer with $\gcd(a,l)=1$. Let $r_1, r_2, \cdots, r_{\phi(l)}$ be the positive integers $\le l$ and relatively prime to $l$. Then the least residues of the integers $ar_1,ar_2,\cdots, ar_{\phi(l)}$ modulo $k$ are a permutation of the integers  $r_1, r_2, \cdots, r_{\phi(l)}$.\end{lem}
 
We need the following  conventions for further developments. 
For each integer $l\ge1$,  $\nu_2(l)$  denotes the maximum power of $2$ dividing $l$. For any odd prime power $q$, let $s=\nu_2(q-1)\ge1$ and $m=\nu_2(q^2-1)$, then  $m-s=\nu_2(q+1)\ge1$. Readily note that  (i) $m=s+1$ if and only if $q\equiv1\pmod 4$ and (ii) $s=1$ if and only if $q\equiv3\pmod 4$.  

\begin{lem} \label{cyclo2n} Let $q$ and $d$ be odd integer  such that $\gcd(q,d)=1$. If  $\beta_{2^k}$ is a primitive $2^k$th root of unity in $\mathbb{F}_{q^2}^*$. Then, for $2\le k\le m-1$, the complete factorization of cyclotomic polynomial $\Phi_{2^k}(x)$ over $\mathbb{F}_{q^2}$ is given by:  \begin{eqnarray} \label{beta}
\Phi_{2^k}(x)=
\displaystyle{\prod_{i=1}^{2^{k-2}}(x\pm\beta_{2^k}^{d(2i-1)})}.\end{eqnarray}   \end{lem}
\begin{proof}  Let $d$ be an odd integer and  integer $k\ge2$. For a fixed $k$, from Lemma \ref{res}, the least reduced residue system of the integers  $d,3d,\cdots, (2^{k}-1)d$ modulo $2^k$ is a permutation of the integers $1,3,\cdots, 2^{k}-1$. Therefore,  a collection of elements $\beta_{2^k}^{d}, \beta_{2^k}^{3d},\cdots, \beta_{2^k}^{d(2^{k}-1)}$  is  a  rearrangement of  the elements $\beta_{2^k}, \beta_{2^k}^{3}, \cdots, \beta_{2^k}^{2^{k}-1}$  in some order in $\mathbb{F}_{q^2}^*$ with $\beta_{2^k}^{d(2i-1)}\neq\beta_{2^k}^{d(2j-1)}$ for $1\le i<j\le 2^{k-1}$. For any  $2\le k\le m-1$ and $q$ is odd, we get $2^k|(q^2-1)$. Since $\beta_{2^k}^{d(2i-1)}$ for $1\le i\le 2^{k-1}$ are $2^{k-1}$ distinct elements in $\mathbb{F}_{q^2}^*$, so we find that \begin{eqnarray*}\label{full}\Phi_{2^k}(x)=\prod_{i=1}^{2^{k-1}}(x-\beta_{2^k}^{d(2i-1)}).\end{eqnarray*} Using the fact\begin{center}$
\beta_{2^k}^{d(2(2^{k-2}+i)-1)}=
\beta_{2^k}^{d(2^{k-1})+d(2i-1)}=-\beta_{2^k}^{d(2i-1)}$ for $1\le i\le 2^{k-2}$, \end{center} the factorization of $\Phi_{2^k}(x)$ can be viewed as follow: \begin{eqnarray*}\Phi_{2^k}(x)&=&\prod_{i=1}^{2^{k-2}}(x-\beta_{2^k}^{d(2i-1)})\prod_{i=2^{k-2}+1}^{2^{k-1}}(x-\beta_{2^k}^{d(2i-1)})\\&=&\prod_{i=1}^{2^{k-2}}(x\pm\beta_{2^k}^{d(2i-1)}).\end{eqnarray*} \end{proof}
The following remark to Lemma \ref{cyclo2n} is easily deduced.
\begin{rem}\label{rephi} If  $q\equiv1\pmod 4$ and $\alpha_{2^k}$ is a primitive $2^k$th root of unity in $\mathbb{F}_q^*$. Then  $\beta_{2^k}=\alpha_{2^k}$ for $2\le k\le m-1=s$.  By Lemma \ref{cyclo2n} it follows that \begin{eqnarray} \label{fact2s}
\Phi_{2^k}(x)=
\displaystyle{\prod_{i=1}^{2^{k-2}}(x\pm\alpha_{2^k}^{d(2i-1)})}.\end{eqnarray}  
Also if  $q\equiv3\pmod 4$, then  $2^k\nmid(q-1)$, but $2^k|(q+1)$ for $2\le k\le m-1$. In this case  $\Phi_{2^k}(x)$ is the product of  $\phi(2^k)/2$ trinomails of degree $2$ over $\mathbb{F}_q$ (see \cite[Lemma 2.6]{singh}).\end{rem}
 The following three lemmas will be useful  in the proof of Theorem \ref{phi4k+1} and Theorem \ref{phi2nd3}.
 \begin{lem}\label{factm}For any positive  integer $\ell$ relatively prime with $q$, let  $c\in\mathbb{F}_q^*$ such that $c=a^{\ell}$ for some $a\in \mathbb{F}_{q}$. Then  \begin{equation}\label{xd-c}\displaystyle{x^{\ell}-c=\prod_{j=0}^{\ell-1}(x-ab^j)},\end{equation} where $b$ is a primitive $\ell$th root of unity in some extension field  of $\mathbb{F}_q$. 
Further, if $\ell$ is an odd divisor of $q+1$,  the explicit factorization of $x^\ell-c$ over $\mathbb{F}_q$ is given by \begin{equation}\label{xd-c2}x^{\ell}-c=(x-a)\prod_{j=1}^{(\ell-1)/2}(x^2-a(b^j+b^{qj})x+a^2).\end{equation} \end{lem}
\begin{proof}  Since $\gcd(\ell,q)=1$, so there exists a least positive integer $l$ such that $\ell|(q^l-1)$. Let $b$ be a primitive $\ell$th root of unity in $\mathbb{F}_{q^l}$. Let $c\in\mathbb{F}_q^*$ such that $c=a^{\ell}$ for some $a\in\mathbb{F}_q^*$. Note that  $a,ab,\cdots, ab^{\ell-1}$ are $\ell$ distinct elements in $\mathbb{F}_{q^l}^*$ such that $(ab^j)^{\ell}=a^{\ell}(b^{\ell})^j=a^{\ell}=c$ for all $0\le j\le \ell-1$. It follows that  $x-ab^j$ is a  monic factor of   $x^\ell-c$  for every $0\le j\le \ell-1$. Therefore the polynomial $x^\ell-c$ is given by the product of polynomials $x-a, x-ab,\cdots, x-ab^{\ell-1}$ over $\mathbb{F}_{q^l}$.

 Further, let $\ell$ be an odd divisor of $q+1$ with $\gcd(q,\ell)=1$. Clearly $\ell\nmid(q-1)$. Then $b^j\in \mathbb{F}_{q^2}^*$ for all $0\le j\le \ell-1$. For any $1\le j\le \ell-1$, we note that $b^j\notin\mathbb{F}_q$. On contrary, assume that $b^j\in\mathbb{F}_q$ for some $1\le j\le \ell-1$, then $qj\equiv j\pmod \ell$. Since $\ell\nmid(q-1)$, so $j\equiv0\pmod \ell$, which is not possible because   $1\le j\le \ell-1$. This proves that $b^j\notin\mathbb{F}_q$ for all $1\le j\le \ell-1$. 

Further, from Expression \ref{xd-c}, it follows that \begin{eqnarray*}x^{\ell}-c&=&(x-a)\prod_{j=1}^{(\ell-1)/2}(x-ab^j)\prod_{j=(\ell+1)/2}^{\ell-1}(x-ab^{j})\\&=&(x-1)\prod_{j=1}^{(\ell-1)/2}(x-ab^j)\prod_{j=1}^{(\ell-1)/2}(x-ab^{\ell-j}).\end{eqnarray*}  Also, for all $1\le j\le (\ell-1)/2$,  since $b^{qj}=b^{-j}=b^{\ell-j}$ and $a(b^j+b^{qj})\in\mathbb{F}_q$ as $(a(b^j+b^{qj}))^q=a^q(b^{qj}+b^{q^2j})=a(b^{qj}+b^j)$, so that the minimal polynomial of $b^j$  over $\mathbb{F}_q$ is $(x-ab^j)(x-ab^{qj})=x^{2}-a(b^j+b^{qj})x+a^2$. Therefore,   \begin{center}$\displaystyle{x^{\ell}-c=
(x-a)\prod_{j=1}^{(\ell-1)/2}(x^2-a(b^j+b^{qj})x+a^2)}$. \end{center} This completes the proof.
 \end{proof}

 In the following results we introduce two recursive sequences which are to be useful later on  in describing the coefficients of irreducible factors of $\Phi_{2^k}(x^d)$ over $\mathbb{F}_q$.
 \begin{lem} \label{gammad} Let $d$ be an odd divisor of $q+1$ and  $\gamma$ be a primitive $d$th root of unity in  $\mathbb{F}_{q^2}$. Let $\delta_{j}=\gamma^j+\gamma^{qj}$ and $\theta_{j}=\gamma^j-\gamma^{qj}$ for $0\le j\le d-1$. Then the recursive sequences $(\delta_j)_{j\ge0}$ of length $(d+1)/2$ in $\mathbb{F}_q$ and $(\theta_j)_{j\ge0}$ of length $d$in $\mathbb{F}_{q^2}$  are given by:   \begin{center}$\delta_j=\delta_1\delta_{j-1}-\delta_{j-2}$ with  $\delta_0=2$ and $\delta_1=\gamma+\gamma^{-1}$ and \end{center} 
\begin{center}$\theta_j=\delta_1\theta_{j-1}-\theta_{j-2}$ with   $\theta_0=0$ and $\theta_1=\gamma-\gamma^{-1}$.\end{center} satisfying  $\delta_{d-j}=d_j$, $\theta_{d-j}=-\theta_j$ and $\delta_j^2-2=\theta_j^2+2=\delta_{2j}$  for $0\le j\le (d-1)/2$.  
 \end{lem}
\begin{proof} Let $\gamma$ be a primitive $d$th root of unity in $\mathbb{F}_{q^2}^*$  and $d|(q+1)$. Let $\delta_j=\gamma^j+\gamma^{qj}$  and $\theta_j=\gamma^j-\gamma^{qj}$ for $0\le j\le d-1$.  Since $\gamma^{qj}=\gamma^{-j}=\gamma^{d-j}$ for all $0\le j\le (d-1)/2$, we get that $\delta_j=\delta_{d-j}$ and $\theta_{j}=-\theta_{d-j}$ for $1\le j\le (d-1)/2$ with $\delta_0=2$ and $\theta_0=0$. It is easy to verify that $\delta_j=\delta_1\delta_{j-1}-\delta_{j-2}$ and  $\theta_j=\delta_1\theta_{j-1}-\theta_{j-2}$ for  $j\ge 2$. Also,  for any $0\le j\le (d-1)/2$, we establish the relations $\delta_j^2-2=\gamma^{2j}+\gamma^{2qj}=\delta_{2j}$ and $\theta_j^2+2=\gamma^{2j}+\gamma^{2qj}=\delta_{2j}$.
\end{proof}

\begin{lem}\label{delta} Let $\Delta_d=\{\delta_j:1\le j\le (d-1)/2\}$.  Then $\Delta_d$ has $(d-1)/2$ distinct elements of $\mathbb{F}_q^*$ and the mapping $\sigma_d:\Delta_d\rightarrow \Delta_d$ defined by $\sigma_d(\delta_j)=\delta_{2j}$ is a permutation.\end{lem}
\begin{proof}  Let  $\delta_{j_1}$, $\delta_{j_2}$ in $\Delta_d$ such that $\delta_{j_1}=\delta_{j_2}$, where $1\le j_1, j_2\le (d-1)/2$. By Lemma \ref{gammad}, we have $j_2=d-j_1$, a contradiction to the choice of $j_1$ and $j_2$. Thus $\Delta_d$ has $(d-1)/2$ distinct elements.  By Lemma \ref{gammad}, $\delta_j^2-2=\delta_{2j}$ for $1\le j\le (d-1)/2$. If $1\le 2j\le (d-1)/2$, then $\delta_{2j}\in\Delta_d$, and if $(d+1)/2\le 2j\le d-1$, then $1\le t\le (d-1)/2$ for $t=d-2j$ and hence $\delta_{2j}=\delta_{d-t}=\delta_t\in\Delta_d$. It follows that  $\delta_{2j}\in \Delta_d$  for every $1\le j\le (d-1)/2$.  Now, define $\sigma_d:\Delta_d\rightarrow \Delta_d$ such that $\sigma_d(\delta_j)=\delta_{2j}$. Obviously,  $\sigma_d$ is well defined mapping. Further, let $\sigma_d(\delta_{j_1})=\sigma(\delta_{j_2})$ for $1\le j_1, j_2\le (d-1)/2$, then $\delta_{2j_1}=\delta_{2j_2}$.  This gives either $\delta_{j_1}=\delta_{j_2}$ or $\delta_{j_1}=-\delta_{j_2}$. If $\delta_{j_1}=-\delta_{j_2}$, then $\gamma^{j_1+j_2}=-1$ or $\gamma^{j_1-j_2}=-1$, which is not possible because the order of $\gamma$ is odd. Thus for any $1\le j_1, j_2\le (d-1)/2$, $\sigma_d(\delta_{j_1})=\sigma_d(\delta_{j_2})$ implies that $\delta_{j_1}=\delta_{j_2}$ and hence $\sigma_d$ is a permutation.\end{proof}

\section{ Factorization of  $\Phi_{2^n}(x^d)\in\mathbb{F}_q[x]$}
  In this section, we study a factorization of decomposable cyclotomic polynomials of the type $\Phi_{2^n}(x^d)$ over $\mathbb{F}_q$, where $d$ is an odd  divisor of $q+1$ and integer $n\ge1$.  We begin with the more general factorization:  
 \begin{lem}\label{club}Let  $q$ and $d$ be odd integers.  Then, for any integer $n\ge2$,  
\begin{eqnarray*} \label{club}
x^{2^nd}-1=\left\{ \begin{array}{lcl}
 \displaystyle{(x^{d}-1)(x^{d}+1)\prod_{k=2}^{n}\Phi_{2^{k}}(x^d)} & \mbox{for} & 2\le n\le m-1\\\displaystyle{(x^{2^{m-1}d}-1)\prod_{r=0}^{n-m}\Phi_{2^m}(x^{2^rd})}& \mbox{for}
& n\ge m.
\end{array}\right.\end{eqnarray*}\end{lem} 
\begin{proof} For any integer $n\ge 2$, it is trivial to note that \begin{center}$x^{2^n}-1=\displaystyle{(x-1)(x+1)\prod_{2\le k\le n}\Phi_{2^k}(x)}$.\end{center} For a moment if $n\ge m$, then we express \begin{center}$x^{2^n}-1=\displaystyle{(x^{m-1}-1)\prod_{m\le k\le n}\Phi_{2^k}(x)}$.\end{center} Let $n=m+r$, then by Lemma \ref{cyclotomic},  $\Phi_{2^n}(x)=\Phi_{2^m}(x^{2^r})$, where $0\le r\le n-m$. By applying $x\rightarrow x^d$, we can easily meet the desired conclusion.\end{proof} In view of the above, our idea for obtaining the explicit factorization of $x^{2^nd}-1$ over $\mathbb{F}_q$ is to factorize decomposable cyclotomic polynomials  of the form $\Phi_{2^k}(x^{d})$ of smaller orders into the product of irreducible factors over $\mathbb{F}_q$. Note, $\Phi_1(x^d)=x^d-1$, $\Phi_2(x^d)=\Phi_1(-x^d)=x^d+1$ and for $2\le k\le m-1$, from Lemma \ref{cyclo2n},  the cyclotomic polynomial $\Phi_{2^k}(x^d)=x^{2^{k-1}d}+1$  over $\mathbb{F}_{q^2}$ can be written as:  \begin{eqnarray} \label{beta1}
\Phi_{2^k}(x^d)=
\displaystyle{\prod_{i=1}^{2^{k-2}}(x^d\pm\beta_{2^k}^{d(2i-1)})}.\end{eqnarray} 
Since  $2^k|(q+1)$ if $q\equiv3\pmod 4$, and $2^k|(q-1)$ if $q\equiv1\pmod 4$, so we have two distinguishable cases as follows:
 \begin{thm}\label{phi4k+1} Let $q\equiv1\pmod 4$ and $q\equiv-1\pmod d$. Then \begin{enumerate}
\item[(i)] If $2\le k\le s$,  the factorization of $\Phi_{2^k}(x^d)$ into $2^{k-2}(d+1)$  irreducible factors over $\mathbb{F}_q$ is given by:
\begin{eqnarray}\label{phi} 
\Phi_{2^k}(x^d)= \displaystyle{\prod_{\substack{i=1\\1\le j\le (d-1)/2}}^{2^{k-2}}(x\pm\alpha_{2^k}^{2i-1})(x^2\pm
\alpha_{2^{k}}^{2i-1}\delta_jx+
\alpha_{2^{k-1}}^{2i-1})}. \end{eqnarray}  
 \item[(ii)]  For any integer $r\ge1$, the factorization of $\Phi_{2^{s+r}}(x^d)$ into $2^{s-1}d$ irreducible factors over $\mathbb{F}_q$ is given by:
\begin{eqnarray}\label{phim} 
\Phi_{2^{s+r}}(x^d)=\displaystyle{\prod_{\substack{i=1\\1\le j\le (d-1)/2\\0\le l\le 1}}^{2^{s-2}}(x^{2^r}\pm\alpha_{2^k}^{2i-1})(x^{2^r}\pm
\alpha_4^{l}\beta_{2^{s+1}}^{2i-1}\theta_jx^{2^{r-1}}-
\alpha_2^l\alpha_{2^{s}}^{2i-1})}\nonumber\\.\end{eqnarray} 
 
\end{enumerate} 
 \end{thm}
\begin{proof}  If $q\equiv1\pmod 4$, then $m=s+1$. It follows that $2^k|(q-1)$  and $\beta_{2^k}=\alpha_{2^k}$ for $2\le k\le s$ and $2^{s+1}\nmid(q-1)$ as $s=\nu_2(q-1)$. Now we have the following two cases:
\\{\bf{Case (i)}} For $2\le k\le s$,  from Expression \ref{beta1}, a factorization of $\Phi_{2^k}(x^d)$  over $\mathbb{F}_q$ is given by: \begin{eqnarray*} 
\Phi_{2^k}(x^d)=\displaystyle{\prod_{i=1}^{2^{k-2}}(x^d\pm\alpha_{2^k}^{d(2i-1)})}.\end{eqnarray*} 
 For each $1\le i\le 2^{k-2}$, from  Expression \ref{xd-c2}, we  find the following factorization of $x^d\pm\alpha_{2^k}^{d(2i-1)}$ into the product of $d+1$ irreducible factors over $\mathbb{F}_q$ \begin{eqnarray}\label{alpha}x^d\pm\alpha_{2^k}^{d(2i-1)}&=&
(x\pm\alpha_{2^k}^{2i-1})\prod_{j=1}^{(d-1)/2}(x^2\pm\alpha_{2^k}^{2i-1}\delta_j x+\alpha_{2^{k-1}}^{2i-1})\end{eqnarray} and hence the number of irreducible factors of $\Phi_{2^k}(x^d)$  over $\mathbb{F}_q$ is $2^{k-2}(d+1)$.\newline
{\bf{Case (ii)}}  If $q\equiv1\pmod 4$, then $2^{s+1}\nmid(q-1)$, but $2^{s+1}|(q^2-1)$.   Applying  Lemma \ref{cyclotomic},  Expression \ref{phi} yields a factorization of $\Phi_{2^{s+1}}(x^d)=\Phi_{2^s}(x^{2d})$ over $\mathbb{F}_q$ with $2^{s-2}(d+1)$ factors  as follows: \begin{eqnarray}\label{4}\Phi_{2^{s+1}}(x^d)=\prod_{i=1}^{2^{s-2}}\bigg(
(x^2\pm\alpha_{2^s}^{2i-1})\prod_{j=1}^{(d-1)/2}(x^4\pm\alpha_{2^s}^{2i-1}\delta_jx^2+\alpha_{2^{s-1}}^{2i-1}).\bigg)\nonumber\\\end{eqnarray} 
Since $\alpha_{2^s}^{2i-1}$ is a non-square element in $\mathbb{F}_q^*$, so $x^2-\alpha_{2^s}^{2i-1}$ is  irreducible over $\mathbb{F}_q$. Now, we check the irreducibility of   trinomials $x^{4}-\alpha_{2^s}^{2i-1}\delta_{j}x^2+
\alpha_{2^{s-1}}^{2i-1}$, where  $1\le i\le 2^{s-2}$ and $1\le j\le (d-1)/2$ by applying conditions of Lemma \ref{tirr}. Let $f(x)=x^{2}-\alpha_{2^s}^{2i-1}\delta_{j}+
\alpha_{2^{s-1}}^{2i-1}$, where fixed $i$ and $j$ are as $1\le i\le 2^{s-2}$ and $1\le j\le (d-1)/2$. Then  $f(x^{2})=x^{4}-\alpha_{2^s}^{2i-1}\delta_jx^{2}+
\alpha_{2^{s-1}}^{2i-1}$. Observe that $f(x)$ is the minimal polynomial of $\alpha_{2^s}^{2i-1}\gamma^j$ over $\mathbb{F}_q$ so the order of $f(x)$ is $2^sd_1$, where $d_1=d/\gcd(j,d)$. Recall $m=\nu_2(q^2-1)$ and $s=\nu_2(q-1)$. Since $q\equiv1\pmod 4$, so that $m=s+1$ and hence   $\displaystyle{\gcd(2,(q^2-1)/2^sd_1)=2}$. This goes against  the condition (ii) of Lemma \ref{tirr}. Therefore $f(x^{2})$ is reducible over  $\mathbb{F}_q$. Using the similar argument given in \cite[Theorem 3.14]{lidl}, $x^{2}+\alpha_{2^s}^{2i-1}$ is irreducible while $f(-x^2)=x^{4}+\alpha_{2^s}^{2i-1}\delta_{j}x^2+
\alpha_{2^{s-1}}^{2i-1}$ is reducible over $\mathbb{F}_q$.

 It follows that Expression \ref{4} contains $2^{s-1}$ irreducible binomials and $2^{s-2}(d-1)$ reducible trinomials of degree $4$ over $\mathbb{F}_q$. Our aim is to split these $2^{s-2}(d-1)$ trinomials into a product of $2^{s-1}(d-1)$ trinomials of degree $2$ over $\mathbb{F}_q$.

  First consider $\displaystyle{\prod_{j=1}^{(d-1)/2}
\big(x^{4}-
\alpha_{2^s}^{2i-1}\delta_jx^{2}+
\alpha_{2^{s-1}}^{2i-1}\big)}$,  a  product of $2^{s-2}(d-1)/2$ trinomials from  Expression \ref{4}. Applying Lemma  \ref{delta}, we have that \begin{eqnarray*}\label{red}\prod_{j=1}^{(d-1)/2}
\big(x^{4}-
\alpha_{2^s}^{2i-1}\delta_jx^{2}+
\alpha_{2^{s-1}}^{2i-1}\big)&=&\prod_{j=1}^{(d-1)/2}
\big(x^{4}-
\alpha_{2^s}^{2i-1}\sigma_d(\delta_j)x^{2}+
\alpha_{2^{s-1}}^{2i-1}\big)\nonumber \\&=&\prod_{j=1}^{(d-1)/2}\big(x^{4}-
\alpha_{2^s}^{2i-1}(\theta_j^2+2)x^{2}+
\alpha_{2^{s-1}}^{2i-1}\big)\nonumber \\&=&\prod_{j=1}^{(d-1)/2}\big((x^{2}-
\alpha_{2^{s}}^{2i-1})^2-
(\beta_{2^{s+1}}^{2i-1}\theta_jx)^2\big)
\nonumber \\&=&\prod_{j=1}^{(d-1)/2}(x^2\pm
\beta_{2^{s+1}}^{2i-1}\theta_jx-
\alpha_{2^{s}}^{2i-1}).\end{eqnarray*}
Applying the substitution  $x\rightarrow \alpha_4x$, it is quite easy to see

\begin{eqnarray*}\label{red1}\prod_{j=1}^{(d-1)/2}
\big(x^{4}+
\alpha_{2^s}^{2i-1}\delta_jx^{2}+
\alpha_{2^{s-1}}^{2i-1}\big)&=&\prod_{j=1}^{(d-1)/2}(x^2\pm
\alpha_4\beta_{2^{s+1}}^{2i-1}\theta_jx+
\alpha_{2^{s}}^{2i-1}).\end{eqnarray*} Combining the above two expressions, we get \begin{eqnarray*}\label{red2}\prod_{j=1}^{(d-1)/2}
\big(x^{4}\pm
\alpha_{2^s}^{2i-1}\delta_jx^{2}+
\alpha_{2^{s-1}}^{2i-1}\big)&=&\prod_{\substack{j=1\\0\le l\le 1}}^{(d-1)/2}(x^2\pm
\alpha_4^l\beta_{2^{s+1}}^{2i-1}\theta_jx-\alpha_2^l
\alpha_{2^{s}}^{2i-1}).\end{eqnarray*}
Since $\beta_{2^{s+1}}^{q-1}=-1$ and $\theta_j^q=-\theta_j$ so that $(\beta_{2^{s+1}}^{2i-1}\theta_j)^q=
(\beta_{2^{s+1}}^{2i-1})^{q-1}\beta_{2^{s+1}}^{2i-1}(-\theta_j)=\beta_{2^{s+1}}^{2i-1}\theta_j$ and $(\alpha_4\beta_{2^{s+1}}^{2i-1}\theta_j)^q=
(\beta_{2^{s+1}}^{2i-1}\theta_j)^q=
\beta_{2^{s+1}}^{2i-1}\theta_j$, hence   $\beta_{2^{s+1}}^{2i-1}\theta_j$ and $\alpha_4\beta_{2^{s+1}}^{2i-1}\theta_j$ are elements of $\mathbb{F}_q^*$.  Using the above discussion,  $\Phi_{2^{s+1}}(x^d)$ can be written as a product of  $2^{s-2}(2+(\frac{d-1}{2}) 4)=2^{s-1}d$ irreducible factors over $\mathbb{F}_q$ as follows:
\begin{eqnarray}\label{phim} 
\Phi_{2^{s+1}}(x^d)=\displaystyle{\Phi_{2^s}(x^2)\prod_{\substack{i=1\\1\le j\le (d-1)/2\\0\le l\le 1}}^{2^{s-2}}(x^2\pm
\alpha_4^{l}\beta_{2^{s+1}}^{2i-1}\theta_jx-
\alpha_2^l\alpha_{2^{s}}^{2i-1})}\end{eqnarray} Applying $x\rightarrow x^{2^{r-1}}$ in Expression \ref{phim}, we get\begin{eqnarray}\label{phir} 
\Phi_{2^{s+r}}(x^d)=\displaystyle{\Phi_{2^s}(x^{2^r})\prod_{\substack{i=1\\1\le j\le (d-1)/2\\0\le l\le 1}}^{2^{s-2}}(x^{2^r}\pm
\alpha_4^{l}\beta_{2^{s+1}}^{2i-1}\theta_jx^{2^{r-1}}-
\alpha_2^l\alpha_{2^{s}}^{2i-1}).}\end{eqnarray} Let $g(x)=x^{2}-\beta_{2^{s+1}}^{2i-1}\theta_jx+\alpha_{2^{s}}^{2i-1}$, where $1\le i\le 2^{s-2}$ and $1\le j\le (d-1)/2$. Observe that $g(x)$ is the minimal polynomial of $\beta_{2^{s+1}}^{2i-1}\gamma^j$ over $\mathbb{F}_q$ and hence the order of $g(x)$ is $2^{s+1}d_2$, where $d_2$ is a divisor of $d$. Since $s+1=\nu_2(q^2-1)$, $2\nmid(q^2-1)/2^{s+1}$ and hence $\gcd(2^r, (q^2-1)/2^{s+1}d_2)=1$ for every $r\ge1$. From Lemma \ref{tirr} and \cite[Theorem 3.14]{lidl} for $s\ge2$, it follows that $g(x^{2^{r-1}})$, $g(-x^{2^{r-1}})$ are  irreducible over $\mathbb{F}_q$ of same orders $2^{s+r}d_2$.  Since $s\ge2$, so  the order of  $g(\pm\alpha_4x^{2^{r-1}})$ is $2^{s+r}d_2$, It follows that  the irreducibility of $g(\pm\alpha_4x^{2^{r-1}})$ over $\mathbb{F}_q$ can be verified  in similar fashion as discussed for $g(\pm x^{2^{r-1}})$. Therefore,   Expression \ref{phir} has $2^{s-1}d$ irreducible factors over $\mathbb{F}_q$ for every $r\ge1$.  This completes the proof.\end{proof}
\begin{Exaexa}  For $q=29$, $n=3$ and $d=5$, we obtain $\alpha_4=12$, $\beta_8=2\sqrt{3}$, $\gamma=-3+15\sqrt{3}$, $\gamma^{-1}=-3-15\sqrt{3}$, $\theta_1=\sqrt{3}$ and $\theta_2=-6\sqrt{3}$.  By Theorem  \ref{phi4k+1}, the explicit factorization of $\Phi_{8}(x^5)$ over $\mathbb{F}_{29}$ is given by  \begin{eqnarray*}\Phi_{8}(x^5)&=&\Phi_{8}(x)\prod_{\substack{1\le j\le 2\\0\le l\le 1}}(x^2\pm
(12)^{l}(2\sqrt{3})\theta_jx-
(-1)^l12)\\&=&\Phi_{8}(x)(x^2\pm
(2\sqrt{3})\theta_1x-
12)(x^2\pm
(2\sqrt{3})\theta_2x-
12)\\&&(x^2\pm
(12)(2\sqrt{3})\theta_1x+12)(x^2\pm
(12)(2\sqrt{3})\theta_2x+12)\\&=&(x^2+12)(x^2+17)(x^2\pm
6x+17)(x^2\pm 7x+17)\\&&(x^2\pm
14x+12)(x^2\pm
3x+12).\end{eqnarray*} As we know that $\Phi_{40}(x)=\dfrac{\Phi_8(x^5)}{\Phi_8(x)}$ and $\Phi_8(x)=(x^2-12)(x^2+12)$, so that $$\Phi_{40}(x)=(x^2\pm
6x+17)(x^2\pm 7x+17)(x^2\pm
14x+12)(x^2\pm
3x+12).$$ Moreover, for any integer $n\ge 3$, \begin{eqnarray*}\Phi_{2^n\dot 5}(x)&=&\Phi_{40}(x^{2^{n-3}})\\&=&(x^{2^{n-2}}\pm
6x^{2^{n-3}}+17)(x^{2^{n-2}}\pm 7x^{2^{n-3}}+17)\\&&(x^{2^{n-2}}\pm
14x^{2^{n-3}}+12)(x^{2^{n-2}}\pm
3x^{2^{n-3}}+12).\end{eqnarray*}\end{Exaexa}
  
 \begin{thm}\label{phi2nd3} Let $q\equiv3\pmod 4$ and  $q\equiv-1\pmod d$. 
Then the factorization of $\Phi_{2^k}(x^d)$ into  $2^{k-2}d$ irreducible factors  over $\mathbb{F}_q$ is given by:
  \begin{eqnarray}\label{phi4k+3} 
\Phi_{2^k}(x^d)=\left\{ \begin{array}{lcl}
 \displaystyle{\prod_{\substack{i=1\\0\le j\le d-1}}^{2^{k-2}}(x^2-\theta_{i,j,k}x+1)} & \mbox{for} & 2\le k\le m-1\\\displaystyle{\prod_{\substack{i=1\\0\le j\le d-1}}^{2^{m-3}}(x^{2^{k-m+1}}\pm\theta_{i,j}x^{2^{k-m}}-1)}& \mbox{for}
& k\ge m\ge 3.
\end{array}\right.\nonumber\\\end{eqnarray} 
 where  $\theta_{i,j,k}=\beta_{2^k}^{2i-1}\gamma^j+
\beta_{2^k}^{q(2i-1)}
\gamma^{qj}$  for $1\le i\le 2^{k-2}$ and $\theta_{i,j}=\beta_{2^m}^{2i-1}\gamma^j+
\beta_{2^m}^{q(2i-1)}
\gamma^{qj}$  for $1\le i\le 2^{m-3}$. The complete factorization of $\Phi_{2^k}(x)\in\mathbb{F}_q[x]$ is  given in \cite[Lemma 2.6]{singh}.
 \end{thm}
\begin{proof}
 Let $q\equiv3\pmod 4$ and $q\equiv-1\pmod d$.  Then $q^2\equiv1\pmod d$ and $m=\nu_2(q^2-1)$ and from Expression \ref{beta1}, the cyclotomic polynomial $\Phi_{2^k}(x^d)$  over $\mathbb{F}_{q^2}$ in the following form: \begin{eqnarray*}\Phi_{2^k}(x^d)&=&\prod_{i=1}^{2^{k-2}}(x^d-\beta_{2^k}^{d(2i-1)})(x^d+\beta_{2^k}^{d(2i-1)}).\end{eqnarray*}
 
 Using the permutation $i\mapsto 2^{k-2}-i+1$, we have that \begin{eqnarray*}\prod_{i=1}^{2^{k-2}}(x^{d}+\beta_{2^{k}}^{d(2i-1)})&=&\prod_{i=1}^{2^{k-2}}(x^{d}-\beta_{2^{k}}^{d(-2i+1)}).\end{eqnarray*} 
 For each fixed $k$, where $2\le k\le m-1$, since $2^k|(q+1)$, we have $\beta_{2^k}^{q+1}=1$ and hence $\beta_{2^k}^{qd(2i-1)}=\beta_{2^{k}}^{d(-2i+1)}$. Thus the factorization  of $\Phi_{2^k}(x^d)$ over $\mathbb{F}_{q^2}$ reduces to \begin{eqnarray*}\Phi_{2^k}(x^d)&=&\prod_{i=1}^{2^{k-2}}(x^d-\beta_{2^k}^{d(2i-1)})(x^d-\beta_{2^k}^{qd(2i-1)}),\end{eqnarray*} 
In view of Expression \ref{xd-c} of Lemma \ref{factm}, we have \begin{eqnarray*}\Phi_{2^k}(x^d)&=&\prod_{\substack{i=1\\0\le j\le d-1}}^{2^{k-2}}(x-\beta_{2^k}^{2i-1}\gamma^j)(x-\beta_{2^k}^{q(2i-1)}\gamma^j).\end{eqnarray*}
 Since $\gcd(q,d)=1$,  so the least reduced residue system $\{qj\pmod d:0\le j\le d-1\}$ is a rearrangement of   integers $0,1,\cdots, d-1$. Therefore  \begin{eqnarray*}\Phi_{2^k}(x^d)=\prod_{\substack{i=1\\0\le j\le d-1}}^{2^{k-2}}(x-\beta_{2^k}^{2i-1}\gamma^j)(x-\beta_{2^k}^{q(2i-1)}\gamma^{qj}).\end{eqnarray*}For  any $1\le i\le 2^{k-2}$, the minimal polynomial of $\beta_{2^k}^{2i-1}\gamma^j$  is $$x^2-(\beta_{2^k}^{2i-1}\gamma^j+\beta_{2^k}^{q(2i-1)}
\gamma^{qj})x+\beta_{2^k}^{(q+1)(2i-1)}.$$
 Therefore, the irreducible factorization  of $\Phi_{2^k}(x^d)$ over $\mathbb{F}_q$ is \begin{eqnarray}\label{m-1}\Phi_{2^k}(x^d)&=&\prod_{\substack{i=1\\0\le j\le d-1}}^{2^{k-2}}(x^2-\theta_{i,j,k}x+1)\end{eqnarray} with $\theta_{i,j,k}=\beta_{2^k}^{2i-1}\gamma^j+
\beta_{2^k}^{q(2i-1)}
\gamma^{qj}$ for $1\le i\le 2^{k-2}$, $1\le j\le d-1$ and $2\le k\le m-1$. Since $\Phi_{2^m}(x^d)=\Phi_{2^{m-1}}(x^{2d})$. It follows that, from Expression
 \ref{m-1}, \begin{eqnarray*}\Phi_{2^m}(x^d)=\prod_{\substack{i=1\\0\le j\le d-1}}^{2^{m-3}}(x^4-\theta_{i,j,m-1}x^2+1).\end{eqnarray*}  Since $d$ is odd, so the least residue system $\{2j:0\le j\le d-1\}$ is a rearrangement of integers $0,1,\cdots, d-1$. Therefore \begin{eqnarray*}\Phi_{2^m}(x^d)&=&\prod_{\substack{i=1\\0\le j\le d-1}}^{2^{m-3}}(x^4-\theta_{i,2j,m-1}x^2+1).\end{eqnarray*}
Now for $m\ge3$, denote $\theta_{i,j}=\beta_{2^m}^{2i-1}\gamma^j+
\beta_{2^m}^{q(2i-1)}
\gamma^{qj}$ for $1\le i\le 2^{m-3}$ and $0\le j\le d-1$. Since $m-1=\nu_2(q+1)\ge2$ for $q\equiv3\pmod 4$, so that $\beta_{2^m}^{q+1}=-1$. Note that $\theta_{i,j}^2=\theta_{i,2j,m-1}-2$ for every $1\le i\le 2^{m-3}$ and $0\le j\le d-1$. It follows that $$x^4-\theta_{i,2j,m-1}x^2+1=(x^2-1)^2-\theta_{i,j}^2x^2=
(x^2-\theta_{i,j}x-1)(x^2+\theta_{i,j}x-1).$$  Therefore \begin{eqnarray}\Phi_{2^m}(x^d)=\prod_{\substack{i=1\\0\le j\le d-1}}^{2^{m-3}}(x^2-\theta_{i,j}x-1)(x^2+\theta_{i,j}x-1). \end{eqnarray} This proves the result.
\end{proof}

\section{The factorization of $x^{2^nd}-1$ when $d|(q+1)$}
 This section determines  the explicit factorization of $x^{2^nd}-1$ over $\mathbb{F}_{q}$, where $d$ is an odd divisor of $q+1$ and integer $n\ge1$.   We begin with the factorization of $x^d-1$ and $x^d+1$ over $\mathbb{F}_q$.

\begin{lem}\label{d} Let $\mathbb{F}_q$ be a finite field and, let $d$ be an odd integer such that $d|(q+1)$. Then
 \begin{eqnarray*} \label{ed}
x^{d}-1=\displaystyle{(x-1)\prod_{j=1}^{(d-1)/2}(x^2-\delta_jx+1)}\end{eqnarray*}  and \begin{eqnarray*} \label{ed}
x^{d}+1=\displaystyle{(x+1)\prod_{j=1}^{(d-1)/2}(x^2+\delta_jx+1)},\end{eqnarray*}
   where $\delta_j=\gamma^j+\gamma^{qj}\in\mathbb{F}_q^*$ for every $1\le j\le d-1$ and $\gamma$ be a primitive $d$th root of unity in $\mathbb{F}_{q^2}^*$.  \end{lem}
\begin{proof}   The result follows directly  by  substituting $b=\gamma$ and $c=\pm1$ in Expression \ref{xd-c2} of  Lemma \ref{factm}.\end{proof}

 In view of Lemma \ref{d}, \begin{eqnarray*} \label{ed1}
x^{2d}-1=\displaystyle{(x\pm1)\prod_{j=1}^{(d-1)/2}(x^2\pm\delta_jx+1)}. \end{eqnarray*}  where $q\equiv-1\pmod d$.  Further, for each $2\le n\le m-1$, since $2^n|(q-1)$ if $q\equiv1\pmod 4$ and $2^n|(q+1)$ if $q\equiv3\pmod 4$,  therefore the factorization of $x^{2^nd}-1$ over $\mathbb{F}_q$ will be determined for $q\equiv1\pmod 4$ and $q\equiv3\pmod 4$ separately. 
\subsection{Case: $q\equiv1\pmod 4$}
\begin{thm}\label{q2nd4k+1} Let $\mathbb{F}_q$ be a finite field with  $q\equiv1\pmod 4$ elements and   $d$  be an odd integer such that $q\equiv-1\pmod d$.     
\begin{itemize}
\item[(i)] If $2\le n\le s$,  then the factorization of $x^{2^nd}-1$ into the product of $2^{n-1}(d+1)$ monic irreducible factors over $\mathbb{F}_{q}$ is given by:
 \begin{eqnarray*} 
x^{2^nd}-1\nonumber&=&\displaystyle{(x-1)(x+1)\prod_{\substack{i=1\\2\le k\le n}}^{2^{k-2}}(x\pm\alpha_{2^k}^{2i-1})}\times\\&&
\prod_{\substack{i=1\\1\le j\le (d-1)/2\\2\le k\le n}}^{2^{k-2}}\bigg((x^2\pm\delta_jx+1)
(x^2\pm\alpha_{2^k}^{2i-1}
\delta_jx+\alpha_{2^{k-1}}^{2i-1})\bigg).
\end{eqnarray*}  
 \item[(ii)] If $n>s\ge2$, the factorization of $x^{2^nd}-1$ into the product of $  2^{s-1}((n-s+1)d+1)$ monic irreducible factors over $\mathbb{F}_{q}$ is given by:
 \begin{eqnarray*} 
x^{2^nd}-1&=&(x^{2^sd}-1)
\prod_{\substack{i=1\\1\le r\le n-s\\1\le j\le (d-1)/2\\0\le l\le 1}}^{2^{s-2}}\bigg((x^{2^r}\pm\alpha_{2^s}^{2i-1})\big(x^{2^{r}}\pm
\alpha_4^l\beta_{2^{s+1}}^{2i-1}\theta_jx^{2^{r-1}}-
\alpha_2^l
\alpha_{2^{s}}^{2i-1}\big)\bigg)\end{eqnarray*}  where    the complete factorization of $x^{2^sd}-1$ is given in  Case (i). 
\end{itemize}\end{thm}
\begin{proof} If $q\equiv1\pmod 4$, then $m=s+1$. By  Lemma \ref{club}, the factorization of $x^{2^nd}-1$ over $\mathbb{F}_q$ can be viewed as: \begin{eqnarray*} \label{club1}
x^{2^nd}-1=\left\{ \begin{array}{lcl}
 \displaystyle{(x^{d}-1)(x^{d}+1)\prod_{k=2}^{n}\Phi_{2^{k}}(x^d)} & \mbox{for} & 2\le n\le s\\\displaystyle{(x^{2^{s}d}-1)\prod_{r=1}^{n-s}\Phi_{2^{s}}(x^{2^{r}d})}& \mbox{for}
& n>s\ge2.
\end{array}\right.\end{eqnarray*} 
{\bf{Case (i)}} For any $2\le n\le s$, by Lemma \ref{d} and   Theorem \ref{phi4k+1}, the factorization of  $x^{2^nd}-1$ into   the product of  $2^{n-1}(d+1)$ monic  irreducible factors can be  obtained.\\ 
{\bf{Case (ii)}} For $n>s$, by Theorem \ref{phi4k+1},  we obtain the required form of  the factorization of $x^{2^nd}-1$ over $\mathbb{F}_q$ with the number of monic irreducible factors of $x^{2^nd}-1$ over $\mathbb{F}_q$ is the sum of the number of irreducible factors of $x^{2^sd}-1$ and $2^{s-2}(n-s)2+2^{s-2}(n-s)\dfrac{(d-1)}{2}\dot 4$. In view of case (i),  the number of monic irreducible factors of $x^{2^sd}-1$ over $\mathbb{F}_q$ is  $2^{s-1}(d+1)$ and hence the number of monic irreducible factors of $x^{2^nd}-1$ over $\mathbb{F}_q$ is \begin{eqnarray*}
2^{s-1}(d+1)+2^{s-1}(n-s)d=2^{s-1}((n-s+1)d+1).
\end{eqnarray*}
This completes the proof.
\end{proof}
\subsection{Case: $q\equiv3\pmod 4$}

\begin{thm}\label{q2nd4k-1}Let $q\equiv3\pmod 4$ and $d$ be an odd integer such that $q\equiv-1\pmod d$. Then   \begin{itemize}

\item[(i)] If $2\le n\le m-1$, the factorization of $x^{2^nd}-1$ into the product of $2^{n-1}d+1$ irreducible factors  over $\mathbb{F}_{q}$  is given by: 
\begin{eqnarray*} 
x^{2^nd}-1&=&(x-1)(x+1)
\prod_{j=1}^{(d-1)/2}(x^2\pm\delta_jx+1)
\prod_{\substack{i=1\\2\le k\le n\\0\le j\le d-1}}^{2^{k-2}}(x^2-\theta_{i,j,k}x+1),\end{eqnarray*} where $\theta_{i,j,k}=\beta_{2^k}^{2i-1}\gamma^j+
\beta_{2^k}^{q(2i-1)}\gamma^{qj}$ for  $1\le i\le 2^{k-2}$ and  $0\le j\le d-1$.
\item[(ii)] If $n\ge m\ge 3$,  the factorization of $x^{2^nd}-1$ into the product of $2^{m-2}d(n-m+2)+1$ irreducible factors  over $\mathbb{F}_{q}$  is given by:   \begin{eqnarray*} 
x^{2^nd}-1=\displaystyle{(x^{2^{m-1}d}-1)
\prod_{\substack{i=1\\0\le r\le n-m\\0\le j\le d-1}}^{2^{m-3}}(x^{2^{r+1}}\pm\theta_{i,j}x^{2^r}-1)}.\end{eqnarray*}
 where  $\theta_{i,j}=\beta_{2^m}^{2i-1}\gamma^j+
\beta_{2^m}^{q(2i-1)}\gamma^{qj}$ for $1\le i\le 2^{m-3}$, $0\le j\le d-1$. The complete factorization of $x^{2^{m-1}d}-1$ over $\mathbb{F}_q$ is a special case of the factorization given above in (i).
\end{itemize} 
\end{thm}
\begin{proof} Let $d$ be an odd divisor of $q+1$.  Then, by Lemma \ref{club}, 
\begin{eqnarray*} 
x^{2^nd}-1=\left\{ \begin{array}{lcl}
 \displaystyle{(x^{d}-1)(x^{d}+1)\prod_{k=2}^{n}\Phi_{2^{k}}(x^d)} & \mbox{for} & 2\le n\le m-1\\\displaystyle{(x^{2^{m-1}d}-1)\prod_{r=0}^{n-m}\Phi_{2^m}(x^{2^rd})}& \mbox{for}
& 3\le m\le n.
\end{array}\right.\end{eqnarray*}
{\bf{Case (i)}} For any  $2\le n\le m-1$, the factorization of $x^{2^nd}-1$ over $\mathbb{F}_q$ follows from Lemma \ref{d} and Theorem \ref{phi2nd3} such as:  \begin{eqnarray*}x^{2^nd}-1&=&(x-1)(x+1)\prod_{j=1}^{(d-1)/2}(x^2\pm\delta_jx+1)\prod_{\substack{i=1\\2\le k\le n\\0\le j\le d-1}}^{2^{k-2}}(x^2-\theta_{i,j,k}x+1),  \end{eqnarray*} where    $\theta_{i,j,k}=\beta_{2^k}^{2i-1}\gamma^j+
\beta_{2^k}^{q(2i-1)}
\gamma^{qj}$ for $1\le i\le 2^{k-2}$, $0\le j\le d-1$ and $2\le k\le m-1$.   Further, the number of irreducible factors of $x^{2^nd}-1$ is $$1+1+(d-1)+\sum_{2\le k\le n}2^{k-2}d=(d+1)+d(2^{n-1}-1)=2^{n-1}d+1.$$ 
{\bf{Case (ii)}} For any $n\ge m\ge 3$,  the factorization of $x^{2^nd}-1$ over $\mathbb{F}_q$ follows from Case (i) for $n=m-1$ and Theorem \ref{phi2nd3} such as: \begin{eqnarray*}x^{2^nd}-1&=&(x^{2^{m-1}d}-1)
\prod_{\substack{i=1\\0\le j\le d-1\\0\le r\le n-m}}^{2^{m-3}}(x^{2^{r+1}}\pm\theta_{i,j}x^{2^r}-1).\end{eqnarray*} By Lemma \ref{tirr}, it is quite easy to verify that trinomials $x^{2^{r+1}}\pm\theta_{i,j}x^{2^r}-1$ over $\mathbb{F}_q$  are irreducible over $\mathbb{F}_q$. The number of irreducible factors of $x^{2^nd}-1$ is $$2^{m-2}d+1+2^{m-2}d(n-m+1)=2^{m-2}d(n-m+2)+1.$$ This finishes the proof.\end{proof}

\subsection*{Worked examples}

\begin{Exaexa} Let $q=29$ and $n\ge3$. Then $s=2$, $d=15$, $\alpha_{4}=12$, $\alpha_{8}=2\sqrt{3}$, $\gamma=2+\sqrt{3}$, $\gamma^{-1}=2-\sqrt{3}$, $\delta_1=4$ and   $\theta_1=2\sqrt{3}$. By Lemma \ref{gammad}, we obtain  $\delta_2=\delta_1\delta_1-2=14$, $\delta_3=\delta_1\delta_2-\delta_1=23$, $\delta_4=4\cdot 23-14=20$, $\delta_5=28$, $\delta_6=5$ and  $\delta_7=21$. Further, again by Lemma \ref{gammad},  $\theta_2=\delta_1\theta_1=8\sqrt{3}$, $\theta_3=\delta_1\theta_2-\theta_1=\sqrt{3}$, $\theta_4=-4\sqrt{3}$, $\theta_5=12\sqrt{3}$, $\theta_6=-6\sqrt{3}$, $\theta_7=-7\sqrt{3}$. By Theorem \ref{q2nd4k+1}  the explicit factorization of $x^{2^n15}-1$ over $\mathbb{F}_{29}$ is given by: 
 \begin{eqnarray*} 
x^{2^n15}-1&=&(x^{60}-1)
\prod_{\substack{j=1\\1\le r\le n-2\\0\le l\le 1}}^{7}(x^{2^r}\pm\alpha_4))(x^{2^{r}}\pm\alpha_4^l\alpha_{8}\theta_jx^{2^{r-1}}-\alpha_2^l
\alpha_{4}\big)\end{eqnarray*} with  \begin{eqnarray*} 
x^{60}-1&=&(x-1)(x+1)(x-\alpha_4)(x+\alpha_4)
\prod_{\substack{j=1}}^{7}(x^2\pm\delta_jx+1)\big(x^2\pm\alpha_{4}\delta_jx+
\alpha_{2}\big).\end{eqnarray*} 
 
\end{Exaexa}
\begin{Exaexa}Let $q=59$ and $n=4$. Then $m=3$ and $d=15$.  As $q\equiv3\pmod 8$,  $\sqrt{2}\notin\mathbb{F}_{59}^*$, we  first notice that $\beta_4=18\sqrt{2}$, $\beta_4^{-1}=-18\sqrt{2}$, $\beta_8=18+\dfrac{1}{\sqrt{2}}$, and hence $\beta_8^{-1}=-18+\dfrac{1}{\sqrt{2}}$. Further, $\gamma=-5-22\sqrt{2}$,  $\gamma^{-1}=-5+22\sqrt{2}$, $\gamma^2=-10-16\sqrt{2}$  and  $\{\delta_j:1\le j\le 7\}=\{-10,-20,-26, -15, -1,25,-13\}$.   It is easy to see that $\beta_4, \beta_8, \gamma$ are elements in $\mathbb{F}_{59^2}\setminus\mathbb{F}_{59}$ as $\sqrt{2}\notin\mathbb{F}_{59}^*$.  By Theorem \ref{q2nd4k-1}, the factorization of $x^{240}-1$ over $\mathbb{F}_{59}$ is given by    \begin{eqnarray*} 
x^{240}-1=\displaystyle{(x^{60}-1)\prod_{\substack{0\le r\le 1\\0\le j\le 14}}(x^{2^{r+1}}\pm\theta_{1,j}x^{2^r}-1)},\end{eqnarray*}
 where  $\theta_{1,j}=\beta_{8}\gamma^j-
\beta_{8}^{-1}\gamma^{-j}$ for  $0\le j\le 14$ and \begin{eqnarray*} 
x^{60}-1=\displaystyle{(x^{4}-1)
\prod_{j=1}^{7}(x^2\pm\delta_jx+1)
\prod_{\substack{1\le j\le 14}}(x^2-\theta_{1,j,2}x+1)}\end{eqnarray*} with $\theta_{1,j,2}=\beta_{4}\gamma^j+
\beta_{4}^{-1}\gamma^{-j}$ for every  $1\le j\le 14$.  \end{Exaexa}
\section*{Acknowledgments} The author would like to thank the anonymous referee of an earlier version that allowed us to discover a gap in the proof of Theorem 3.2.

\end {document}